\documentclass[12pt,a4paper]{amsart}
\usepackage{mathrsfs}
\usepackage{amsfonts}
\usepackage{amssymb}
\usepackage[top=35mm, bottom=35mm, left=30mm, right=30mm]{geometry}
\usepackage{mathptmx}
\usepackage{xcolor}
\usepackage{hyperref}

\newtheorem{thm}{Theorem}[section]
\newtheorem{lem}[thm]{Lemma}

\newtheorem{prop}[thm]{Proposition}

\theoremstyle{definition}
\newtheorem{de}[thm]{Definition}
\newtheorem{rem}[thm]{Remark}
\newtheorem*{rem*}{Remark}
\newtheorem{ques}{Question}
\numberwithin{equation}{section}
\begin{document}
\title{On multi-transitivity with respect to a vector}

\author[Z.~Chen]{Zhijing Chen}
\address[Z.~Chen]{Department of Mathematics, Zhongshan University, Guangzhou 510275, P. R. China}
\email{chzhjing@mail2.sysu.edu.cn}
\author[J.~Li]{Jian Li}
\thanks{Corresponding author: Jian Li (lijian09@mail.ustc.edu.cn)}
\address[J.~Li]{Department of Mathematics, Shantou University, Shantou, Guangdong 515063, P.R. China}
\email{lijian09@mail.ustc.edu.cn}
\author[J.~L\"u]{Jie L\"U}
\address[J.~L\"u]{School of Mathematics, South China Normal University, Guangzhou 510631, P. R. China}
\email{ljie@scnu.edu.cn}

\thanks{The first and third author were supported by National Nature Science Funds of China (Grant  no.~11071084).
The second author was supported in part by STU Scientific Research Foundation for Talents (NTF12021) and
the National Natural Science Foundation of China (Grants no.\@11071231, 11171320).}
\date{\today}

\begin{abstract}
A topological dynamical system $(X,f)$ is said to be multi-transitive
if for every $n\in\mathbb{N}$ the system $(X^{n}, f\times f^{2}\times \dotsb\times f^{n})$ is transitive.
We introduce the concept of multi-transitivity with respect to a vector and show that
multi-transitivity can be characterized by the hitting time sets of  open sets,
answering a question proposed by  Kwietniak and Oprocha [On weak mixing, minimality and weak  disjointness of
all iterates,  Erg. Th. Dynam. Syst., 32 (2012), 1661--1672].
We also show that multi-transitive systems are  Li-Yorke chaotic.
\end{abstract}
\maketitle


\section{Introduction}
A topological dynamical system $(X,f)$ is called \emph{transitive}
if for every two non-empty open subsets $U$ and $V$ of $X$
there exists a positive integer $k$ such that $U\cap f^{-k}(V)$ is not empty.
In his seminal paper~\cite{Furstenberg-1967}, Furstenberg showed that $2$-fold transitivity
(also known as weak mixing) implies $n$-fold transitivity for every positive integer $n$, that is
\begin{thm}
Let $(X,f)$ be a topological dynamical system. If $(X\times X,f\times f)$ is transitive, then
for every $n\in\mathbb{N}$,
the product system $(X\times X\times\dotsb \times X, f\times f\times\dotsb \times f)$ ($n$-times) is also transitive.
\end{thm}
In 2010, Moothathu~\cite{Moothathu-2010} introduced the notion of multi-transitivity.
A dynamical system $(X,f)$ is called \emph{multi-transitive} if for every $n\in\mathbb{N}$,
the product system $(X^n,f\times f^2\times\dotsb\times f^n)$ is transitive.
He showed that for a minimal system multi-transitivity is equivalent to weak mixing, and
he also asked whether there are implications between multi-transitivity and weak mixing
for general (not necessarily minimal) systems.

Recently, Kwietniak and Oprocha~\cite{D-Kwietniak-P-Oprocha-2010} showed that in general
there is no connection between multi-transitivity and weak mixing
by constructing examples of weakly mixing but non-multi-transitive
and multi-transitive but non-weakly mixing systems. They proposed the following problem.

\begin{ques}
Is there any non-trivial characterization of multi-transitive weakly mixing systems?
\end{ques}

In this paper, we will give a positive answer to this question  and
show that multi-transitivity can be characterized by the hitting time sets of open sets.
To this end, we first generalize the concept of multi-transitivity
by introducing multi-transitivity with respect to a vector.
More specifically, for a vector $\mathbf{a}=(a_1,a_2,\dotsc,a_r)$ of positive integers,
a dynamical system $(X,f)$ is called \emph{multi-transitive with respect to the vector $\mathbf{a}$}
if the system  $(X^r,f^{a_1}\times f^{a_2}\times \dotsb \times f^{a_r})$  is transitive.
In Section 3 we show some basic properties of multi-transitivity with respect to a vector.
In Section 4 we define a new kind of Furstenberg family generated by a vector, and
show that the hitting time sets of open sets of a multi-transitive system with respect to a vector can be
characterized by this kind of Furstenberg family.
In Section 5, we study multi-transitivity with respect to a sequence, and show that
multi-transitive systems with respect to a difference set are strongly scattering.
Finally, we show that multi-transitive systems are  Li-Yorke chaotic.

\section{Preliminaries}
In this paper, the sets of all non-negative integers and positive natural numbers
are denoted by $\mathbb{Z}_+$ and $\mathbb{N}$ respectively.
For $r\in\mathbb{N}$, denote $\mathbb{N}^r=\mathbb{N}\times \mathbb{N}\times \dotsb\times \mathbb{N}$ ($r$-copies).
\subsection{Topological dynamics}
A \emph{topological dynamical system} is a pair $(X, f)$,
where $X$ is a compact metric space with a metric $d$ and $f:X\rightarrow X$ is a continuous map.
Let $f^{0}=id$ (the identity map) and for $n\in\mathbb{N}$,  let $f^n=f^{n-1}\circ f$.
For $n\in\mathbb{N}$, we can also define the $n$-th product $(X^n,f^{(n)})$ of $(X,f)$,
where $X^n=X\times X\times \dotsb\times X$ ($n$-times) and $f^{(n)}=f\times f\times \dotsb\times f$ ($n$-times).

Let $(X, f)$ be a dynamical system.
For two subsets $U$, $V$ of $X$, we define the \emph{hitting time set of $U$ and $V$} by
\[N(U,V)=\{n\in\mathbb{N}:\ f^n(U)\cap V\neq\emptyset\}=\{n\in\mathbb{N}:\ U\cap f^{-n}(V)\neq\emptyset\}.\]
We say that $(X,f)$ is \emph{(topologically) transitive} if for every
two non-empty open subsets $U$ and $V$ of $X$, the hitting time set $N(U,V)$ is non-empty;
\emph{totally transitive} if $(X,f^n)$ is transitive for any $n\in\mathbb{N}$;
\emph{weakly mixing} if the product system $(X\times X,f\times f)$ is transitive;
\emph{strongly mixing} if for every
two non-empty open subsets $U$ and $V$ of $X$, the set $N(U,V)$ is cofinite,
that is, there exists $N\in\mathbb{N}$ such that $\{N,N+1,\dotsc\}\subset N(U,V)$.

For $x\in X$, denote the orbit of $x$ by  $Orb_f(x)=\{x,f(x),f^2(x),\dotsc,f^n(x),\dotsc\}$.
A point $x\in X$ is called a \emph{transitive point}
if the closure of the orbit of $x$ equals $X$, i.e., $\overline{Orb_f(x)}=X$.
We denote the set of all transitive points of $(X,f)$ by $Trans(X, f)$.
A dynamical system $(X,f)$ is called \emph{minimal} if it has no proper closed invariant subsets, that is,
if $K\subset X$ is non-empty, closed and $f(K)\subset K$, then $K=X$. It is easy to see that
a system $(X,f)$ is minimal if and only if every point in $X$ is a transitive point, i.e., $Trans(X, f)=X$.

A dynamical system $(X, f)$ is called an {\it E-system} if it is transitive and there is an invariant Borel
probability measure $\mu$ on $X$ with full support, i.e.,
$supp(\mu)=\{x\in X:$ for every neighborhood $U$ \mbox{of} $x$, $\mu(U)>0\}=X$.

Recall that two dynamical systems $(X,f)$ and $(Y,g)$ are called \emph{weakly disjoint}
if their product system $(X\times Y,f\times g)$ is transitive.
Using weak disjointness, we can define several kinds of dynamical systems \cite{BHM00,HY04}.
A dynamical system $(X, f)$ is called \emph{scattering} if it is weakly disjoint from any  minimal system;
\emph{strongly scattering} if it is weakly disjoint from any E-system;
\emph{mildly mixing} if it is weakly disjoint from any transitive system.

\subsection{Furstenberg families}
Let $\mathcal{P}$ denote the collection of all subsets of $\mathbb{N}$.
A subset $\mathcal F$ of $\mathcal P$ is called a \emph{Furstenberg family} (or just a \emph{family}),
if it is hereditary upward, i.e.,
\[\text{$F_1 \subset F_2$ and $F_1\in\mathcal F$ imply $F_2\in\mathcal{F}$.} \]
 A family $\mathcal{F}$ is called \emph{proper} if it is a non-empty
proper subset of $\mathcal P$,
i.e., it is neither empty nor all of $\mathcal P$.
Any non-empty collection $\mathcal A$ of subsets of $\mathbb{N}$ naturally generates a family
\[\mathcal F(\mathcal A) = \{F\subset\mathbb{N}:\ A \subset F \text{ for some } A\in\mathcal {A}\}.\]
Let $\mathcal F_{inf}$ be the family of all infinite subsets of $\mathbb{N}$,
and $\mathcal F_{cf}$ be the family of all cofinite subsets of $\mathbb{N}$.

A subset $F$ of $\mathbb{N}$ is called \emph{thick} if it contains arbitrarily long runs of positive integers,
i.e., for every $n\in\mathbb{N}$ there exists some $a_n\in\mathbb{N}$ such that $\{a_n, a_n+1,\dotsc, a_n+n\}\subset F$.
The family of all thick sets is denoted by $\mathcal F_{t}$.

\subsection{Topological dynamics via Furstenberg families}
The idea of using families to describe dynamical properties goes back at least to Gottschalk and Hedlund \cite{GH55}.
It was developed further by Furstenberg \cite{F81}.
There are serval ways to classify transitive systems via Furstenberg familes,
see~\cite{A97}, \cite{G04}, \cite{HY04}, \cite{W-Huang-X-Ye-2005} and~\cite{L2011}.
Here we just recall one of them.
It allows us to classify transitive systems by the hitting time sets of open sets.

Let $(X,f)$ be a dynamical system and $\mathcal{F}$ be a family.
The system $(X,f)$ is called \emph{$\mathcal F$-transitive}
if for every two non-empty open subsets $U,V$ of $X$, the hitting time set $N(U,V)$ is in $\mathcal F$;
\emph{$\mathcal F$-mixing} if $(X\times X, f\times f)$ is $\mathcal F$-transitive.

\begin{lem}[\cite{A97,Furstenberg-1967}]
Let $(X, f)$ be a dynamical system and $\mathcal F$ be a family.
Then
\begin{enumerate}
\item $(X,f)$ is weakly mixing if and only if it is $\mathcal F_{t}$-transitive.
\item $(X,f)$ is strongly mixing if and only if it is $\mathcal F_{cf}$-transitive.
\item $(X,f)$ is $\mathcal F$-mixing if and only if it is $\mathcal F$-transitive and weakly mixing.
\end{enumerate}
\end{lem}

\section{Multi-transitivity with respect to a vector}
In this section, we generalize the concept of multi-transitivity by introducing
multi-transitivity with respect to a vector and show some basic properties of multi-transitivity.

\begin{de}\label{def:muti-trans}
Let $(X, f)$ be a dynamical system and $\mathbf{a}=(a_{1}, a_{2}, \dotsc, a_r)$ be a vector in $\mathbb{N}^r$.
We say that $(X,f)$ is
\begin{enumerate}
\item \emph{multi-transitive with respect to the vector $\mathbf{a}$}
(or briefly \emph{$\mathbf{a}$-transitive})
if the product system $(X^r,  f^{(\mathbf{a})})$ is transitive,
where $f^{(\mathbf{a})}=f^{a_{1}}\times f^{a_{2}}\times \dotsb \times f^{a_{r}}$;
\item \emph{multi-transitive}
if it is multi-transitive with respect to $(1,2,\dotsc,n)$ for any $n\in\mathbb N$;
\item \emph{strongly multi-transitive}
if it is multi-transitive with respect to any vector in $\mathbb{N}^n$ and any $n\in\mathbb{N}$.
\end{enumerate}
\end{de}

It is clear that strong mixing is stronger than strong multi-transitivity, which in turn is stronger than
weak mixing. The authors in~\cite{D-Kwietniak-P-Oprocha-2010} showed that
there is no implication between weak mixing and multi-transitivity by constructing two special spacing shifts,
one is a multi-transitive non-weakly mixing system,
and the other is a weakly mixing non-multi-transitive system.
In fact, for every $m\geq 2$ they constructed a weakly mixing spacing shift
which is multi-transitive with respect to $(1,2\dotsc,m-1)$ but not for $(1,2,\dotsc,m)$.

\begin{lem}\label{lem:mul-tran-eq}
Let $(X,f)$ be a dynamical system and $n\in\mathbb{N}$.
Then $(X,f)$ is multi-transitive if and only if $(X,f^n)$ is multi-transitive.
\end{lem}

\begin{proof}
Necessity.
Let $m\in\mathbb{N}$ and $\mathbf{a}=(1,2,\dotsc,m)$.
We want to show that $(X,f^n)$ is $\mathbf{a}$-transitive.
Let $U_1$, $U_2,\dotsc, U_m$, $V_1$, $V_2,\dotsc, V_m$ be non-empty open subsets of $X$.
For $i=1,2,\dotsc,n$, put $U_i'=U_1$, $U_{n+i}'=U_2, \dotsc, U_{(m-1)n+i}'=U_m$,
$V_i'=V_1$, $V_{n+i}'=V_2, \dotsc, V_{(m-1)n+i}'=V_m$.
Let $\mathbf{a}'=(1,2,\dotsc,mn)$.
By the transitivity of $(X^{mn},f^{(\mathbf{a}')})$, we have that
$$N_{f^{(\mathbf{a}')}}
(U_1'\times U_2'\times \dotsb \times U_{mn}',V_1'\times V_2'\times \dotsb \times V_{mn}')\neq\emptyset.$$
But we also have
$$N_{f^{(\mathbf{a}')}}(U_1'\times U_2'\times\dotsb\times U_{mn}',V_1'\times V_2'\times\dotsb\times V_{mn}')
\subset N_{(f^n)^{(\mathbf{a})}}(U_1\times U_2\times\dotsb\times U_{m},V_1\times V_2\times\dotsb\times V_{m}),$$
which implies that  $(X,f^n)$ is $\mathbf{a}$-transitive.

Sufficiency. Let $m\in\mathbb{N}$ and  $\mathbf{a}=(1,2,\dotsc,m)$.
By the fact $(f^{(\mathbf{a})})^n=(f^n)^{(\mathbf{a})}$ and $(X^m,(f^n)^{(\mathbf{a})})$ is transitive,
$(X,f)$ is $\mathbf{a}$-transitive. Thus $(X,f)$ is multi-transitive.
\end{proof}

\begin{prop}\label{prop:s-mul-tran-eq}
Let $(X,f)$ be a dynamical system, $n\in\mathbb{N}$ and $\mathbf{a}=(a_1,a_2,\dotsc,a_r)\in\mathbb{N}^r$.
Then the following conditions are equivalent:
\begin{enumerate}
  \item $(X,f)$ is strongly multi-transitive;
  \item $(X,f^n)$ is strongly multi-transitive;
  \item $(X^r,f^{(\mathbf{a})})$ is strongly multi-transitive;
  \item for every $k\in\mathbb{N}$, $(X^k, f\times f^2\times \dotsb \times f^k)$ is weakly mixing;
  \item $(X,f)$ is weakly mixing and multi-transitive.
\end{enumerate}
\end{prop}
\begin{proof}
(1) $\Rightarrow$ (2) Let $l\in\mathbb{N}$ and $\mathbf{a}'=(a_1',a_2',\dotsc,a_l')\in\mathbb{N}^l$.
As $(X,f)$ is strongly multi-transitive,
$(X,f)$ is multi-transitive with respect to the vector $(na_1',na_2',\dotsc,na_l')$,
which implies that $(X^l, $ $(f^{n})^{(\mathbf{a}')})$ is transitive.
Then $(X,f^n)$ is multi-transitive with respect to the vector $\mathbf{a}'$,
and $(X,f^n)$ is also strongly multi-transitive.

(2) $\Rightarrow$ (3) Let $l\in\mathbb{N}$ and $\mathbf{a}'=(a_1',a_2',\dotsc,a_l')\in\mathbb{N}^l$.
As $(X,f^n)$ is strongly multi-transitive, $(X,f^n)$ is multi-transitive with respect to the vector
$$(a_1a_1',a_2a_1',\dotsc,a_ra_1',a_1a_2',a_2a_2',\dotsc,a_ra_2',\dotsc,a_1a_l',a_2a_l',\dotsc,a_ra_l'),$$
which implies that
$(X^{rl},(f^{(\mathbf{a})})^{(\mathbf{a}')})$ is transitive.
Thus $(X^r,f^{(\mathbf{a})})$ is multi-transitive with respect to the vector $\mathbf{a}'$.

(3) $\Rightarrow$ (4) Let $k\in\mathbb{N}$ and $\mathbf{a}'=(1,1,2,2,\dotsc,k,k)\in\mathbb{N}^{2k}$.
As $(X^r,f^{(\mathbf{a})})$ is strongly multi-transitive,
$(X^r, f^{(\mathbf{a})})$ is multi-transitive with respect to the vector $\mathbf{a}'$,
which implies that
$(X^{2k}, $ $f^{a_1}\times f^{a_1}\times f^{2a_1}\times f^{2a_1}\times\dotsb\times f^{ka_1}\times f^{ka_1})$
is transitive. Therefore $(X^k,f\times f^2\times \dotsb \times f^k)$ is weakly mixing.

(4) $\Rightarrow$ (1)
Let $l\in\mathbb{N}$ and $\mathbf{a}'=(a_1',a_2',\dotsc,a_l')\in\mathbb{N}^l$.
Without loss of generality, assume that $a_i'\leq a_{i+1}'$ for $i=1,2,\dotsc,l-1$.
Let $g =f\times f^2\times\dotsb \times f^{a_l'}$. By the assumption, $(X^{a_l'},g)$ is weakly mixing,
then $(X^{la_l'}, g^{(l)})$ is transitive, which implies that
$(X^l,f^{(\mathbf{a}')})$ is transitive.
Thus $(X,f)$ is multi-transitive with respect to the vector $\mathbf{a}'$.

(5) $\Rightarrow$ (4)
Fix $k\in\mathbb{N}$ and let $U_1,U_2,\dotsc, U_{2k}$,
$V_1,V_2,\dotsc,V_{2k}$ be non-empty open subsets of $X$.
Since $(X,f)$ is weakly mixing,
$(X^{2k},f^{(2k)})$ is transitive and then there exists $m\in\mathbb{N}$
such that
\[U_{2i-1}\cap f^{-m}(U_{2i})\neq \emptyset\text{ and }V_{2i-1}\cap f^{-m}(V_{2i})\neq \emptyset,\ i=1,2,\dotsc, k. \]
Put  $U_i'=U_{2i-1}\cap f^{-m}(U_{2i})$ and $V_i'=V_{2i-1}\cap f^{-m}(V_{2i})$, $i=1,2,\dotsc, k$.
By the transitivity of $f\times f^2\times\dotsb\times f^k$,
there is $n_0\in \mathbb{N}$ such that
\[U_i'\cap f^{-i n_0}(V_i')\neq \emptyset,\ i=1,2,\dotsc, k.\]
That is
\[U_{2i-1}\cap f^{-m}(U_{2i})\cap f^{-in_0}(V_{2i-1}\cap f^{-m}(V_{2i}))\neq \emptyset, i=1,2,\dotsc, k,\]
which implies that
\[U_{2i-1}\cap f^{-in_0}(V_{2i-1})\neq\emptyset\text{ and }
U_{2i}\cap f^{-in_0}(V_{2i})\neq\emptyset, i=1,2,\dotsc, k. \]
Therefore $(X^k,f\times f^2\times \dotsb \times f^k)$ is weakly mixing.

(4) $\Rightarrow$ (5) is obvious.
\end{proof}

Though weakly mixing systems may not be multi-transitive, we can show
that mildly mixing systems and HY-systems are strongly multi-transitive.
Recall that a system $(X,f)$ is called \emph{mildly mixing}
if for any transitive system $(Y,g)$, the product system $(X\times Y,f\times g)$
is transitive.

\begin{prop}
Let $(X,f)$ be a dynamical system.
If $(X,f)$ is mildly mixing, then for every $\mathbf{a}=(a_1,a_2,\dotsc,a_r)\in\mathbb{N}^r$,
the system $(X^r,f^{(\mathbf{a})})$ is also mildly mixing.
In particulary, $(X,f)$ is strongly multi-transitive.
\end{prop}
\begin{proof}
We proceed by induction on the length $r$ of $\mathbf{a}$.
As a base case, when $r=1$ we need to show that for any $k\in\mathbb{N}$, $(X,f^k)$ is also mildly mixing.
Fix a $k\in\mathbb{N}$ and let $(Y,g)$ be a transitive system.
We want to show that  $(X\times Y,f^k\times g)$ is transitive.
Let $U_1$, $U_2$ be two non-empty open subsets of $X$ and $W_1$,  $W_2$ be two non-empty open subsets of $Y$.
Pick a transitive point $y$ of $(Y,g)$.
Let $K=\{0,1,\dotsc,k-1\}$ with discrete topology and $Z=Y\times K$. Define a map $h: Z\to Z$ as
\[h(x,i)=\begin{cases}
  (x,i+1),& i=0,\dots,k-2,\\
  (g(x),0),& i=k-1.
\end{cases}
\]
It is not hard to verify that $(Z, h)$ is a transitive system and $(y,0)$ is a transitive point of $(Z, h)$.
Let $W_1'=W_1\times \{0\}$ and $W_2'=W_2\times \{0\}$.
Since $(X,f)$ is mildly mixing, $(X\times Z,f\times h)$ is transitive.
Then there exists $n\in\mathbb{N}$ such that
\[U_1\cap f^{-n}(U_2)\neq\emptyset \text{ and }
W_1'\cap h^{-n}(W_2')\neq\emptyset.\]
By the construction of $h$, we have $k|n$. Let $m=n/k$. Then
\[U_1\cap (f^k)^{-m}(U_2)\neq\emptyset \text{ and }
W_1\cap g^{-m}(W_2)\neq\emptyset,\]
which  imply  that $(X\times Y,f^k\times g)$ is transitive.

For the inductive step, let $r > 1$ be an integer, and assume that the result holds for $r-1$.
Let $\mathbf{a}=(a_1,\dotsc,a_{r-1},a_r)$ be a vector in $\mathbb{N}^r$ with length $r$.
By the induction hypothesis, we have that $(X^{r-1},f^{(\mathbf{a}')})$ is mildly mixing,
where $\mathbf{a'}=(a_1,\dotsc,a_{r-1})$.
Let $(Y,g)$ be a transitive system.
As $(X,f^{a_r})$ is a mildly mixing system, we have $(X\times Y,f^{a_r}\times g)$ is transitive.
Again, as $(X^{r-1},f^{(\mathbf{a}')})$ is a mildly mixing system,
we have $(X^{r-1}\times X \times Y,f^{(\mathbf{a}')}\times f^{a_r}\times g)$ is transitive,
that is $(X^r\times Y,f^{(\mathbf{a})}\times g)$ is transitive.
Then $(X^r,f^{(\mathbf{a})})$ is mildly mixing.
Thus the result holds for $r$, and this completes the proof.
\end{proof}

Let $(X, f)$ be a dynamical system. We say that $(X, f)$ has
\emph{dense small periodic sets} if for every non-empty open subset $U$ of $X$
there exists a closed subset $Y$ of $U$ and $k \in\mathbb{N}$
such that $Y$ is invariant for $f^ k$ (i.e., $f^k(Y)\subset Y$).
Clearly, if $(X,f)$ has dense periodic points,
then it  has dense small periodic sets.
We say that $(X,f)$ is an \emph{HY-system}
if it is totally transitive and has dense small periodic sets.

\begin{lem}\label{lem:HY-sys-weak-disjoint}
Let $(X, f)$ be a dynamical system.
If $(X,f)$ is an HY-system, then $(X,f)$ is weakly disjoint from any totally transitive system.
\end{lem}
\begin{proof}
Let $(Y,g)$ be a totally transitive system. We want to show that $(X\times Y,f\times g)$ is transitive.
Let $U_1$ and $U_2$ be two non-empty open subsets of $X$,
$V_1$ and $V_2$ be two non-empty open subsets of $Y$.
By the transitivity of $(X,f)$, there exists $l\in\mathbb{N}$ such that $U_1\cap f^{-l}(U_2)\neq\emptyset$.
As $(X,f)$ has dense small period sets,
there exists a closed subset $F$ of $U_1\cap f^{-l}(U_2)$ and $n\in\mathbb{N}$
such that $f^n(F)\subset F$ and then $n\mathbb{N}+l\subset N_{f}(U_1,U_2)$.
Since $g^{-l}(V_2)$ is a non-empty open set of $Y$ and $(Y,g^n)$ is transitive,
there is  $i\in\mathbb{N}$ such that $V_1\cap g^{-ni-l}(V_2)\neq\emptyset$,
i.e., $ni+l\in N_{g}(V_1, V_2)$.
Therefore $N_{f}(U_1,U_2)\cap N_{g}(V_1, V_2)\neq\emptyset$,
which implies that $(X\times Y,f\times g)$ is transitive.
\end{proof}

\begin{prop}\label{prop:HY-SMT}
Let $(X, f)$ be a dynamical system.
If $(X,f)$ is an HY-system, then for every $\mathbf{a}=(a_1,a_2,\dotsc,a_r)\in\mathbb{N}^r$,
the system $(X^r,f^{(\mathbf{a})})$ is an HY-system.
In particular, $(X,f)$ is strongly multi-transitive.
\end{prop}

\begin{proof}
We proceed by induction on the length $r$ of $\mathbf{a}$.
The base case  $r=1$ is obvious since it is easy to check that for every $k\in\mathbb{N}$,
$(X,f^k)$ is also an HY-system.

For the inductive step, let $r > 1$ be an integer, and assume that the result holds for $r-1$.
Let $\mathbf{a}=(a_1,\dotsc,a_{r-1},a_r)$ be a vector in $\mathbb{N}^r$ with length $r$.
We need to show that $(X^r,f^{(\mathbf{a})})$ is an HY-system.
It is clear that $(X^r, f^{(\mathbf{a})})$ has dense small periodic sets.
Then it suffices to show that $(X^r, f^{(\mathbf{a})})$ is totally transitive.
By the induction hypothesis, we have that $(X^{r-1},f^{(\mathbf{a}')})$ is an HY-system,
where $\mathbf{a'}=(a_1,\dotsc,a_{r-1})$. Denote $Y=X^{r-1}$ and $g=f^{(\mathbf{a}')}$.
Fix a $k\in\mathbb{N}$. We need to prove that $(Y\times X, g^k\times f^{ka_{r}})$ is transitive.
By the induction hypothesis and the base case,
both $(Y,g^k)$ and $(X,f^{ka_{r}})$ are HY-systems.
Then by Lemma~\ref{lem:HY-sys-weak-disjoint}, we have that $(Y\times X, g^k\times f^{ka_{r}})$ is transitive.
\end{proof}

\begin{rem}
Here we give a self-contained proof of Proposition~\ref{prop:HY-SMT}.
As pointed out by the reviewer,
it is known that every HY-system is weakly mixing (see~\cite{W-Huang-X-Ye-2005} or~\cite{O10}).
It can be used to simplify proofs of Lemma~\ref{lem:HY-sys-weak-disjoint} and Proposition~\ref{prop:HY-SMT}
since we need only to prove multi-transitivity.
\end{rem}

\section{Characterization of multi-transitivity}
In this section, we will use the hitting time sets of two non-empty open subsets to characterize multi-transitivity.
We first introduce a new kind of Furstenberg family.
\begin{de}
Let $\mathbf{a}=(a_1,a_2,\dotsc,a_r)$ be a vector in $\mathbb{N}^r$.
We define \emph{the family generated by the vector $\mathbf{a}$},
denoted by $\mathcal{F}[\mathbf{a}]$, as
\begin{align*}
\{F\subset\mathbb{N}:\ \text{for every } \mathbf{n}=(n_1,n_2,\dotsc,n_r)\in\mathbb{Z}_+^r,
\text{ there exists }k\in\mathbb{N} \text{ such that }  k\mathbf{a}+\mathbf{n}\in F^r\},
\end{align*}
where $k\mathbf{a}+\mathbf{n}=(ka_1+n_1,ka_2+n_2,\dotsc,ka_r+n_r)$.
\end{de}

Using the family $\mathcal{F}[\mathbf{a}]$, we have the following characterization of multi-transitivity
with respect to $\mathbf{a}$.
\begin{thm}\label{thm:trans-vetor-equ-condi}
Let $(X,f)$ be a dynamical system and $\mathbf{a}=(a_1,a_2,\dotsc,a_r)\in \mathbb{N}^r$.
Then $(X,f)$ is $\mathbf{a}$-transitive if and only if  it is $\mathcal{F}[\mathbf{a}]$-transitive.
\end{thm}
 \begin{proof}
Necessity. Assume that $(X,f)$ is $\mathbf{a}$-transitive.
Let $U$ and $V$ be two non-empty open subsets of $X$.
We want to show that  $N(U,V)\in \mathcal{F}[\mathbf{a}]$.
Let $n_1, n_2,\dotsc, n_{r}\in\mathbb{Z}_+$.
By the transitivity of $(X^r,f^{(\mathbf{a})})$, there exists $k\in\mathbb{N}$ such that
\[k\in N_{f^{(\mathbf{a})}} (U\times U\times\dotsb\times U,
f^{-n_1}(V)\times f^{-n_2}(V)\times \dotsb \times f^{-n_r}(V)).\]
Then $\{a_{1}k+n_1,a_{2}k+n_2, \dotsc, a_rk+n_{r}\}\subset N(U,V)$.
Hence $N(U,V)\in \mathcal{F}[\mathbf{a}]$.

Sufficiency. Assume that $(X,f)$ is $\mathcal{F}[\mathbf{a}]$-transitive.
Let $U_1,U_2,\dotsc,U_r$ and $V_1, V_2,\dotsc,V_r$ be non-empty open subsets of $X$.
Applying the transitivity of $(X,f)$ to the two non-empty open sets $V_{r-1}$ and $V_r$,
we pick an $\ell'_{r-1}\in\mathbb{N}$ such that
\[V_{r-1}\cap f^{-\ell'_{r-1}}(V_r)\neq\emptyset.\]
Now applying the transitivity of $(X,f)$ to the two non-empty open sets $V_{r-2}$ and $V_{r-1}\cap f^{-\ell'_{r-1}}(V_r)$,
we pick an $\ell'_{r-2}\in \mathbb{N}$ such that
\[V_{r-2}\cap f^{-l_{r-2}'}(V_{r-1}\cap f^{-\ell'_{r-1}}(V_r))\neq\emptyset.\]
After repeating this process $(r-1)$ times,
we obtain a sequence $\{\ell'_i\}_{i=1}^{r-1}$ of positive integers and
\begin{align*}
V':=&V_1\cap f^{-\ell'_{1}}\Bigl(V_{2}\cap f^{-\ell'_{2}}\bigl(\dotsb\cap(V_{r-1}\cap f^{-\ell'_{r-1}}(V_r))\bigr)\Bigr)\\
=& V_1\cap f^{-\ell'_{1}}(V_{2})\cap f^{-(\ell'_{1}+\ell'_{2})}(V_{3})\cap\dotsb\cap
 f^{-(\ell'_{1}+\ell'_{2}+\dotsb+\ell'_{r-1})}(V_r)\neq\emptyset.
\end{align*}
Similarly, we can also obtain a sequence $\{\ell_i\}_{i=1}^{r-1}$  with $\ell_i>\ell_i'$ for $i=1,2,\dotsc,r-1$ and
\begin{align*}
U':=&U_1\cap f^{-\ell_{1}}\Bigl(U_{2}\cap f^{-\ell_{2}}\bigl(\dotsb\cap(U_{r-1}\cap f^{-\ell_{r-1}}(U_r))\bigr)\Bigr)\\
=& U_1\cap f^{-\ell_{1}}(U_{2})\cap f^{-(\ell_{1}+\ell_{2})}(U_{3})\cap\dotsb\cap
 f^{-(\ell_{1}+\ell_{2}+\dotsb+\ell_{r-1})}(U_r)\neq\emptyset.
\end{align*}
By the assumption, we have $N_f(U',V')\in\mathcal{F}[\mathbf{a}]$.
Then for the vector $(0, \ell_1-\ell_1', \ell_1+\ell_2-\ell_1'-\ell_2',\dotsc,
\ell_1+\ell_2+\dotsb+\ell_{r-1}-\ell_1'-\ell_2'-\dotsb-\ell_{r-1}')$,
there exists $k\in \mathbb{N}$ such that
\[\{a_1k,a_2k+\ell_1-\ell_1',\dotsc,a_rk+\ell_1+\ell_2+\dotsb+\ell_{r-1}-\ell_1'-\ell_2'-\dotsb-\ell_{r-1}'\}
\subset N_f(U', V').\]
Then
\begin{align*}
  &U_1\cap (f^{a_1})^{-k} (V_1)\neq\emptyset,\\
  &f^{-\ell_1}(U_2)\cap (f^{a_2})^{-k}(f^{-\ell_1}(V_2))\neq\emptyset,\\
  &\dotsc \\
  &f^{-(\ell_1+\ell_2+\dotsb+\ell_{r-1})}(U_r)\cap
  (f^{a_r})^{-k}(f^{-(\ell_1+\ell_2+\dotsb+\ell_{r-1})}(V_r))\neq\emptyset,
\end{align*}
which imply that $k\in N_{f^{(\mathbf{a})}}(U_1\times U_2\times\dotsb\times U_r,
V_1\times V_2\times\dotsb \times V_r)$.
Hence $(X,f)$ is $\mathbf{a}$-transitive.
\end{proof}

Let $\mathcal{F}[\infty]=\bigcap\limits_{r=1}^\infty \mathcal{F}[\mathbf{a}_r]$,
where $\mathbf{a}_r=(1,2,\dotsc,r)$  for $r\in\mathbb{N}$.
It is easy to see that a subset $F$ of $\mathbb{N}$ is in $\mathcal{F}[\infty]$ if and only if
for every $r\in\mathbb{N}$ and every $\mathbf{n}\in \mathbb{Z}_+^r$, there exists $k\in\mathbb{N}$
such that $k\mathbf{a}_r+\mathbf{n}\in F^r$.
We have the following characterizations of multi-transitivity and strong  multi-transitivity.

\begin{thm}\label{thm:multi-trans-equ-condi-infty}
Let $(X,f)$ be a dynamical system. Then
\begin{enumerate}
  \item $(X,f)$ is multi-transitive if and only if it is $\mathcal{F}[\infty]$-transitive;
  \item $(X,f)$ is strongly multi-transitive if and only if it is $\mathcal{F}[\infty]$-mixing.
\end{enumerate}
\end{thm}

\begin{proof}
For every $r\in\mathbb{N}$, put $\mathbf{a}_r=(1,2,\dotsc,r)\in \mathbb{N}^r$.

(1) Necessity. By the definition of multi-transitivity,
we have that $(X,f)$ is $\mathbf{a}_r$-transitive for all $r\in\mathbb{N}$.
Now by Theorem~\ref{thm:trans-vetor-equ-condi},
$(X,f)$ is $\mathcal{F}[\mathbf{a}_r]$-transitive for all $r\in\mathbb{N}$.
Then for every two non-empty open subsets $U,V$ of $X$,
$N(U,V)\in \mathcal{F}[\mathbf{a}_r]$ for all $r\in\mathbb{N}$.
Therefore $N(U,V)\in \mathcal{F}[\infty]$,
which implies that $(X,f)$ is $\mathcal{F}[\infty]$-transitive.

Sufficiency. Since $(X,f)$ is $\mathcal{F}[\infty]$-transitive,
$(X,f)$ is $\mathcal{F}[\mathbf{a}_r]$-transitive for all $r\in\mathbb{N}$.
By Theorem~\ref{thm:trans-vetor-equ-condi},
$(X,f)$ is $\mathbf{a}_r$-transitive for all $r\in\mathbb{N}$.
Hence $(X,f)$ is multi-transitive.

(2) Necessity. Since $(X,f)$ is strongly multi-transitive,
by Proposition~\ref{prop:s-mul-tran-eq} we have that $(X^r,f^{(\mathbf{a}_r)})$
is weakly mixing for every $i\in\mathbb{N}$.
Then $(X^2,f\times f)$ is $\mathbf{a}_r$-transitive for every $r\in\mathbb{N}$.
Thus $(X^2,f\times f)$ is $\mathcal{F}[\infty]$-transitive,
that is $(X,f)$ is $\mathcal{F}[\infty]$-mixing.

Sufficiency. Assume that $(X,f)$ is $\mathcal{F}[\infty]$-mixing.
Then $(X^2,f\times f)$ is $\mathcal{F}[\infty]$-transitive.
So $(X^2,f\times f)$ is $\mathbf{a}_r$-transitive for every $r\in\mathbb{N}$,
that is $(X^r,f^{(\mathbf{a}_r)})$ is weakly mixing for every $r\in\mathbb{N}$.
Now by Proposition~\ref{prop:s-mul-tran-eq},
we have $(X,f)$ is strongly multi-transitive.
\end{proof}

\section{Multi-transitivity with respect to a sequence}
In this section, we introduce the notion of multi-transitivity with respect to a sequence $A$
and show that multi-transitivity with respect to a difference set implies strong scattering.

\begin{de}
Let $(X, f)$ be a dynamical system and $A=(a_i)_{i=1}^\infty$ be a sequence in $\mathbb{N}$.
We say that $(X,f)$ is \emph{multi-transitive with respect to the sequence $A$} (or briefly \emph{$A$-transitive}),
if the product system $(X^\infty, f^{(A)})$ is transitive,
where  $X^\infty$ is endowed with the Tikhonov topology
and $f^{(A)}=f^{a_{1}}\times f^{a_{2}}\times \dotsb \times f^{a_{i}}\times\dotsb$.

It is not hard to verify that $(X,f)$ is $A$-transitive if and only if for every $i\in\mathbb{N}$,
$(X,f)$ is $\mathbf{a}_i$-transitive, where $\mathbf{a}_i=(a_1,a_2,\dotsc,a_i)$.
\end{de}

For a sequence $A=(a_i)_{i=1}^\infty$ in $\mathbb{N}$,
let $\mathcal{F}[A]=\bigcap_{i=1}^\infty \mathcal{F}[\mathbf{a}_i]$, where $\mathbf{a}_i=(a_1,a_2,\dotsc,a_i)$.
Similarly to the proof of Theorem~\ref{thm:multi-trans-equ-condi-infty},
we have the following characterization of multi-transitivity with respect to a sequence.

\begin{prop}
Let $(X,f)$ be a dynamical system and $A=(a_i)_{i=1}^\infty$ be a sequence in $\mathbb{N}$.
Then $(X,f)$ is $A$-transitive if and only if it is $\mathcal{F}[A]$-transitive.
\end{prop}

Recall that a subset $A$ of $\mathbb{N}$ is a \emph{difference set} if
there exists an infinite subset $B$ of $\mathbb{N}$
such that $B-B:=\{a-b: a,b\in B $ and $a>b\}\subset A$, and a dynamical system is \emph{strongly scattering} if
it is weakly disjoint from every E-system.

\begin{thm}\label{thm:diff-set-tran-impl-s-scatter}
Let $(X, f)$ be a dynamical system and $A$ be a difference set.
If $(X,f)$ is $A$-transitive, then $(X,f)$ is strongly scattering.
\end{thm}

\begin{proof}
By the definition of the difference set, there exists an infinite subset $B$ of $\mathbb{N}$ such that $B-B\subset A$.
We can arrange $B$ as an increasing sequence $\{b_i\}_{i=1}^\infty$.
Let $(Y, g)$ be an $E$-system.
Then there is an invariant Borel probability measure $\mu$ of $(Y,g)$ such that $supp(\mu)=Y$.
We first prove the following claim.

\textbf{Claim}: For every two non-empty open subsets
$U$, $V$ of $X$ and every non-empty open subset $W$ of $Y$, $N_f(U,V)\cap N_g(W,W)\neq\emptyset$.
\begin{proof}[Proof of the Claim]
Since $\mu$ is of full support, $\mu(W)=\alpha>0$.
Pick $n\in \mathbb{N}$ with $n>\frac{1}{\alpha}$,
and let $\{a_1,a_2,\dotsc, a_m\}=\{b_1,b_2,\dotsc,b_n\}-\{b_1,b_2,\dotsc,b_n\}\subset B-B\subset A$.
As $(X,f)$ is $A$-transitive,  there exists a $k\in\mathbb{N}$  such that
\[(U\times U\times\dotsb \times U)\cap
\bigl(f^{(\mathbf{a}_m)}\bigr)^{-k}(V\times V\times\dotsb\times V)\neq\emptyset.\]
Then $\{ka_1,ka_2,\dotsc,ka_m\}\subset N_f(U, V)$.

Consider the following $n$ subsets of $Y$:
\[g^{-kb_1}(W),  g^{-kb_2}(W),  \dotsc, g^{-kb_n}(W).\]
Notice that $\mu(g^{-kb_i}(W))=\mu(W)=\alpha$ for $i=1, 2, \dotsc, n$.
Since $\mu(Y)=1$, by the pigeonhole principle, there exist $1\leq i<j\leq n$  such that
$g^{-kb_i}(W)\cap g^{-kb_j}(W)\neq\emptyset$, i.e., $k(b_j-b_i)\in N_g(W, W)$.
Then $N_g(W, W)\cap N_f(U, V)\neq\emptyset$, since $k(b_j-b_i)\in \{ka_1,ka_2,\dotsc,ka_m\}$.
This ends the proof of the Claim.
\end{proof}
Now we show that $(X,f)$ is weakly disjoint from $(Y,g)$.
In other words, we want to show that $(X\times Y,f\times g)$ is transitive.
Let $U,V$ be two non-empty open subsets of $X$ and $W_1,W_2$ be two non-empty open subsets of $Y$.
By the transitivity of $(Y,g)$, there exists $k\in\mathbb{N}$ such that $W_1\cap g^{-k}(W_2)\neq\emptyset$.
Let $W=W_1\cap g^{-k}(W_2)$. It is easy to see that
\[k+N_f(U,f^{-k}(V))\cap N_g(W,W)\subset N_f(U,V)\cap N_g(W_1,W_2).\]
By the Claim, we have $N_f(U,f^{-k}(V))\cap N_g(W,W)\neq\emptyset$, and then
 $N_f(U,V)\cap N_g(W_1,W_2)\neq\emptyset$.
This implies that $(X\times Y,f\times g)$ is transitive.
\end{proof}

\section{Chaos behavior in multi-transitive systems}
Let $(X,f)$ be a dynamical system.
The definition of Li-Yorke chaos is based on ideas in~\cite{T-Li-J-Yorke-1975}.
A pair $(x, y)\in X\times X$ is called \emph{scrambled} if
\[\limsup_{n\to\infty}d(f^{n}(x), f^{n}(y))>0\quad\text{and}\quad
 \liminf_{n\rightarrow \infty}d(f^{n}(x), f^{n}(y))=0.\]
A subset $C\subset X$ is called \emph{scrambled}
if every non-diagonal pair $(x,y)$ in $C\times C$ is scrambled.
The system $(X,f)$ is called \emph{Li-Yorke chaotic}
if there exists an uncountable scrambled set in $X$.

For a positive real number $\delta>0$, a pair $(x, y)\in X\times X$ is called \emph{$\delta$-scrambled} if
\[\limsup_{n\to\infty}d(f^{n}(x), f^{n}(y))>\delta\quad\text{and}\quad
 \liminf_{n\rightarrow \infty}d(f^{n}(x), f^{n}(y))=0.\]
A subset $C\subset X$ is called \emph{$\delta$-scrambled}
if every non-diagonal pair $(x,y)$ in $C\times C$ is $\delta$-scrambled.
The system $(X,f)$ is called \emph{Li-Yorke $\delta$-chaotic}
if there exists an uncountable $\delta$-scrambled set in $X$.

\begin{rem}
It is shown in~\cite{W-Huang-X-Ye-2002} that scattering implies Li-Yorke chaos,
then by Theorem~\ref{thm:diff-set-tran-impl-s-scatter}
multi-transitivity implies Li-Yorke chaos.
But in this section we want to show a little bit more, that is in any multi-transitive system
there exists a dense Mycielski $\delta$-scrambled set.
\end{rem}

Recall that a subset of $X$ is called a \emph{Mycielski set} if it can be  expressed
as an union of countably many Cantor sets. For  convenience we restate here a
version of Mycielski's theorem (\cite[Theorem~1]{M64}) which we shall use.

\begin{thm}\label{thm:Mycielski-thm}
Let $X$ be a compact metric space without isolated points.
Suppose $R$ is a dense $G_\delta$ subset of $X\times X$.
Then there exists a dense  Mycielski subset $K$ of $X$ such that
for every two distinct points $x,y$ in $K$, the pair $(x,y)$ is in  $R$.
\end{thm}

Recall that a dynamical system $(X, f)$ has \emph{sensitive dependence on initial conditions}
if there exists a $\delta>0$ such that for every  $\varepsilon>0$ and $x\in X$,
there is some $y\in X$ with $d(x,y)<\varepsilon$ and $n\in \mathbb{N}$ with
$d(f^{n}(x), f^{n}(y))>\delta$.

\begin{lem}[\cite{Moothathu-2010}]\label{lem:(1,2)-tran-impl-sensitive}
Let $(X, f)$ be a dynamical system with $X$ being infinite.
If $(X,f)$ is multi-transitive with respect to $(1,2)$,
then  $(X,f)$ has sensitive dependence on initial conditions.
\end{lem}

\begin{thm}\label{thm:mul-tran-impl-deta-Mycie-scram}
Let $(X, f)$ be a dynamical system with $X$ having at least two points.
If $(X,f)$ is multi-transitive, then there exists a dense Mycielski $\delta$-scrambled set for some $\delta>0$.
\end{thm}

\begin{proof}
By Theorem~\ref{thm:diff-set-tran-impl-s-scatter}, the system $(X,f)$ is strongly scattering.
Then by~\cite[Theorem~3]{W-Huang-X-Ye-2002}, the proximal relation
\[PR(X,f)=\{(x,y)\in X^2:\ \liminf_{n\to\infty} d(f^n(x),f^n(y))=0\}\]
is a dense $G_\delta$ subset of $X^2$.
By Lemma~\ref{lem:(1,2)-tran-impl-sensitive}, the system $(X,f)$ has sensitive dependence on initial conditions.
Then by~\cite[Asymptotic Theorem]{W-Huang-X-Ye-2002} or~\cite[Theorem 3.4]{AK03},
there exists a $\delta>0$ such that
\[R=\{(x,y)\in X^2:\ \limsup_{n\to\infty} d(f^n(x),f^n(y))>\delta\}\]
contains a dense $G_\delta$ subset of $X^2$.
Now by Theorem~\ref{thm:Mycielski-thm}, there is a dense Mycielski subset $C$ of $X$ such that
for every two distinct points $x,y$ in $K$, we have $(x,y)\in RP(X,f)\cap R$.
Clearly, $C$ is $\delta$-scrambled. This ends the proof.
\end{proof}

\subsection*{Acknowledgements}
The authors would like to thank the anonymous reviewers for
their valuable comments and suggestions to improve the
quality of the paper.

\end{document}